\documentclass{amsart}
\usepackage[latin1]{inputenc}
\usepackage{amsmath,amssymb,mathtools,mathrsfs}
\usepackage[T1]{fontenc}
\usepackage{amsthm}
\usepackage[all]{xy}
\usepackage{enumerate}
\usepackage[hidelinks]{hyperref}

\usepackage{color}
\usepackage{comment}

\newcommand{\Acal}{{\mathcal{A}}}

\newcommand{\Ecal}{{\mathcal{E}}}

\newcommand{\Hcal}{{\mathscr{H}}}

\newcommand{\Lcal}{{\mathscr L}}

\newcommand{\Mcal}{{\mathscr M}}

\newcommand{\Ocal}{{\mathcal{O}}}

\newcommand{\Vcal}{{\mathcal{V}}}

\newcommand{\Xcal}{{\mathscr X}}

\newcommand{\an}{{\mathrm{an}}}

\newcommand{\C}{\mathbb{C}}

\newcommand{\N}{\mathbb{N}}

\newcommand{\Q}{\mathbb{Q}}

\newcommand{\R}{\mathbb{R}}

\newcommand{\Z}{\mathbb{Z}}

\renewcommand{\and}{\operatorname{and}}

\DeclareMathOperator{\en}{E}

\DeclareMathOperator{\env}{{P}}

\DeclareMathOperator{\Spec}{{Spec}}

\DeclareMathOperator{\vol}{ {vol}}

\DeclareMathOperator{\abs}{{|\ |}}

\DeclareMathOperator{\length}{c}

\newcommand{\Xan}{{X^{\an}}}

\newcommand{\Ko}{{\ensuremath{K^\circ}}}
\newcommand{\Koo}{{K^{\circ\circ}}}
\newcommand{\Kt}{{\tilde{K}}}

\newcommand{\cA}{\mathcal{A}}

\newcommand{\cE}{\mathcal{E}}

\newcommand{\cL}{\mathcal{L}}
\newcommand{\cM}{\mathcal{M}}

\newcommand{\cO}{\mathcal{O}}

\newcommand{\cU}{\mathcal{U}}
\newcommand{\cV}{\mathcal{V}}
\newcommand{\cX}{\mathcal{X}}

\renewcommand{\d}{\delta}
\newcommand{\e}{\varepsilon}

\newcommand{\la}{\lambda}

\newcommand{\p}{\psi}

\newcommand{\ie}{{\rm i.e.\ }}

\newcommand{\n}{\|\cdot\|}

\newcommand{\redu}{\operatorname{red}}

\theoremstyle{plain}
\newtheorem{thm}{Theorem}[section]

\newtheorem{prop}[thm]{Proposition}

\newtheorem{lemma}[thm]{Lemma}
\newtheorem{lem}[thm]{Lemma}

\theoremstyle{definition}
\newtheorem{defi}[thm]{Definition}

\theoremstyle{remark}

\newtheorem{exam}[thm]{Example}
\newtheorem{example}[thm]{Example}

\newtheorem*{thmA}{{\bf Theorem A}} 
 \newtheorem*{thmB}{{\bf Theorem B}}

\newtheorem*{corC}{{\bf Corollary C}} 
\newtheorem*{corD}{{\bf Corollary D}}

\theoremstyle{remark}
\newtheorem*{ackn}{Acknowledgment}

\numberwithin{equation}{section} 

\date{\today}

\thanks{S\'ebastien Boucksom was partly supported by the ANR project GRACK. Walter Gubler and Florent Martin were supported by the collaborative research 
	center SFB 1085 funded by the Deutsche Forschungsgemeinschaft.
}

\begin{document}

\title{Differentiability of relative volumes over an arbitrary non-Archimedean field}
\author{S\'ebastien Boucksom}
\address{S.~Boucksom, Centre de Math\'ematiques Laurent Schwartz, Ecole Polytechnique and CNRS, Institut Polytechnique de Paris}
\email{sebastien.boucksom@polytechnique.edu}

\author[W.~Gubler]{Walter Gubler}
\address{W. Gubler, Mathematik, Universit{\"a}t 
	Regensburg, 93040 Regensburg, Germany, ORCID: 0000-0003-2782-5611}
\email{walter.gubler@mathematik.uni-regensburg.de}

\author[F.~Martin]{Florent Martin}
\address{F. Martin, Mathematik, Universit{\"a}t 
	Regensburg, 93040 Regensburg, Germany}
\email{florent.guy.martin@gmail.com}

\begin{abstract}
Given an ample line bundle $L$ on a geometrically reduced projective scheme defined over an arbitrary non-Archimedean field, we establish a differentiability property for the relative volume of two continuous metrics on the Berkovich analytification of $L$, extending previously known results in the discretely valued case. As applications, we provide fundamental solutions to certain non-Archimedean Monge--Amp\`ere equations, and generalize an equidistribution result for Fekete points. Our main technical input comes from determinant of cohomology and Deligne pairings.

	\bigskip
	
\end{abstract}

\maketitle

\setcounter{tocdepth}{1}
\tableofcontents

\section*{Introduction}

In~\cite{BBGZ}, a variational approach to the resolution of complex Monge--Amp\`ere equations was introduced, inspired by the classical work of Aleksandrov on real Monge--Amp\`ere equations and the Minkowski problem. A key ingredient in this approach is a differentiability property for relative volumes, previously established in~\cite{BB}. 

This variational approach was adapted in~\cite{nama} to non-Archimedean Monge--Amp\`ere equations in the context of Berkovich geometry. While most of the results in that paper assumed the non-Archimedean ground field $K$ to be discretely valued and of residue characteristic $0$, the proof of the differentiability property required a stronger algebraicity assumption that was later removed in~\cite{BGJKM}. Building on these results, a version for trivially valued fields was obtained in~\cite{trivval}, with a view towards the study of K-stability~\cite{nakstab}. 

The main result of the present paper establishes the differentiability property over an arbitrary non-Archimedean field. While only one ingredient in the variational approach, it can already be used to construct fundamental solutions to Monge--Amp\`ere equations, and to generalize the results of~\cite{BE} on equidistribution of Fekete points. Our strategy follows overall that of~\cite{BGJKM}, itself inspired by techniques of Abbes--Bouche~\cite{AB} and Yuan~\cite{Yuan} in the context of Arakelov geometry. As in~\cite{BE}, the extra technical input enabling us to deal with possibly non-Noetherian valuation rings is provided by the Deligne pairings machinery.

Working over non-discretely valued fields arises naturally in several contexts. First, Berkovich analytifications over trivially valued fields form a natural setting to study K-stability, as advocated in~\cite{nakstab}. Next, any non-Archimedean field which is non-trivially valued and algebraically closed (such as $\C_p$) is densely valued. Another instance is in Arakelov theory, where computing the relative height of a projective variety $X$ defined over the function field $F$ of an adelically polarized projective variety $B$ over $\Q$ leads naturally to a bunch of non-Archimedean absolute values on $F$ satisfying a product formula. Here, the absolute values over a prime $p$ are induced by Zariski dense points of the Berkovich analytification of $B \otimes \Q_p$, and are usually not discrete. For details about this generalization of Moriwaki's heights, we refer to~\cite[\S 3]{gubler-hertel}.

\subsection*{Differentiability of relative volumes}
In what follows, $K$ denotes an arbitrary (complete) non-Archimedean field, $X$ is a geometrically reduced projective $K$-scheme, and $L$ is an ample line bundle on $X$. Set $n\coloneqq\dim X$, and denote by $X^\an$ the associated Berkovich analytic space. 

The data of a continuous metric $\phi$ on (the analytification of) $L$ induces for each $m\in\N$ a supnorm $\n_{m\phi}$ on the space of sections $H^0(mL)=H^0(X,mL)$. Here and throughout the paper, we use additive notation for line bundles and metrics, see~\S\ref{sec:metrics}. Given a second continuous metric $\p$ on $L$, one defines the \emph{relative volume} of the associated supnorms as 
$$
\vol(\n_{m\phi},\n_{m\p})\coloneqq\log \left(\frac{\det\n_{m\p}}{\det\n_{m\phi}}\right),
$$
where $\det\n_{m\phi}, \det\n_{m\p}$ denote the induced norms on the determinant line $\det H^0(mL)$. 
This notion of relative volume, introduced in~\cite{CMac,BE}, can be described in terms of (virtual) lengths in the discretely valued case as in \cite{BGJKM}. As a consequence of Chen and Maclean's work~\cite{CMac}, it is proved in~\cite[Theorem 9.8]{BE} that the \emph{relative volume} of $\phi,\p$
$$
\vol(L,\phi,\p):=\lim_{m\to\infty}\frac{n!}{m^{n+1}}\vol(\n_{m\phi},\n_{m\p})
$$
exists in $\R$. 

When $K$ is non-trivially valued, a continuous metric on $L$ is called \emph{psh}\footnote{A shorthand for \emph{plurisubharmonic.}} (or \emph{semipositive}) if it can be written as a uniform limit of metrics on $L$ induced by nef models of $L$. This definition, which goes back to the work of Shou-Wu Zhang~\cite{zhang-95}, is not adapted to the trivially valued case, where the trivial metric on $L$ is the only model metric. An alternative description of psh metrics relying on Fubini--Study can however be adopted~\cite{BE,trivval}, the upshot being that a continuous metric $\phi$ on $L$ is psh if and only if it becomes psh after base change to some (equivalently, any) non-Archimedean extension of $K$. Both approaches give the same  psh metrics on $L$ in the non-trivially valued case and the latter works also in the trivially valued case.

For a continuous psh metric $\phi$ on $L$, a positive Radon measure $(dd^c\phi)^n$ on $X^\an$
was constructed by Chambert-Loir~\cite{CL} for $K$ discretely valued; the general case can be obtained from~\cite{gubler-2007a} by base change to an algebraically closed non-trivially valued extension of $K$, or directly from the local approach in~\cite{ChaDuc}. The main result of the present paper is as follows.

\begin{thmA} \it{Let $K$ be an arbitrary non-Archimedean field $K$, $X$ a projective, geometrically reduced $K$-scheme, and $L$ an ample line bundle on $X$. For any continuous psh metric $\phi$ on $L$ and any continuous function $f$ on $X^\an$, we then have} 
	\begin{equation} \label{diff property}
	\frac{d}{dt}\bigg|_{t=0}\vol(L,\phi+tf,\phi)=\int_{X^\an} f\,(dd^c\phi)^n.
	\end{equation}
\end{thmA} 
Such a differentiability property was already predicted by Kontsevich and Tschinkel in their pioneering investigations of non-Archimedean pluripotential theory~\cite{KT}. A version of Theorem A when $L$ is merely nef will be established in a subsequent paper. 

In the discretely valued case, Theorem A was proved in~\cite{BGJKM}, and the present proof follows the same overall strategy. As a first step, we reduce to the case where $K$ is algebraically closed and non-trivially valued, and $\phi=\phi_\Lcal$, $f=\pm\phi_D$ are respectively induced by an ample model $\Lcal$ of $L$ and a vertical effective Cartier divisor $D$, both living on some model $\Xcal$ of $X$. A filtration argument that goes back to Yuan's work~\cite{Yuan} yields an estimate for 
$$
\vol(\n_{m(\phi+f)},\n_{m\phi})
$$
in terms of the \emph{content} $h^0(D,m\Acal|_D)$ of the torsion $K^\circ$-module $H^0(D,m\Acal|_D)$, where $\Acal$ is a certain ample line bundle on $\Xcal$, and the content is a version of the length adapted to the non-Noetherian valuation ring $K^\circ$. The key ingredient is then the asymptotic Riemann--Roch formula 
$$
h^0(D,m\Acal|_D)\sim\frac{m^n}{n!}\int_{X^\an}\phi_D\,(dd^c\phi_\Acal)^n,
$$
which we obtain as a consequence of the results on determinant of cohomology and metrics on Deligne pairings established in~\cite{BE}. 
%
%
\subsection*{Applications to non-Archimedean pluripotential theory}
The \emph{relative Monge--Amp\`ere energy} of two continuous psh metrics $\phi,\p$ on $L$ is defined as 
$$
\en(\phi,\p):=\frac{1}{n+1}\sum_{j=0}^n\int_{X^\an}(\phi-\p)(dd^c\phi)^j\wedge(dd^c\p)^{n-j},
$$
where $\phi-\p$ is a continuous function on $X^\an$, in our additive notation for metrics. Given any other continuous psh metric $\phi'$, we have 
$$
\frac{d}{dt}\bigg|_{t=0}\en\left((1-t)\phi+t\phi',\p\right)=\int_{X^\an}(\phi'-\phi)\,(dd^c\phi)^n, 
$$
which means that $\phi\mapsto\en(\phi,\p)$ is the unique antiderivative of the Monge--Amp\`ere operator $\phi\mapsto(dd^c\phi)^n$ that vanishes at $\p$, and implies the cocycle property
$$
\en(\phi_1,\phi_2)=\en(\phi_1,\phi_3)+\en(\phi_3,\phi_2)
$$
for any three continuous psh metrics $\phi_1,\phi_2,\phi_3$ on $L$. 

\smallskip

Next, the \emph{psh envelope} $\env(\phi)$ of a continuous metric $\phi$ on $L$ is defined as the pointwise supremum of the family of (continuous) psh metrics $\p$ on $L$ such that $\p\le\phi$. We say that \emph{continuity of envelopes} holds for $(X,L)$ if $\env(\phi)$ is continuous, hence also psh, for all continuous metrics $\phi$. As observed in~\cite[Lemma 7.30]{BE}, continuity of envelopes is equivalent to the fact that the usc upper envelope of any bounded above family of psh metrics on $L$ remains psh, a classical property in (complex) pluripotential theory which leads to the natural conjecture that continuity of envelopes holds as soon as $X$ is normal.

At present, continuity of envelopes has been established when $X$ is smooth, and one of the following holds: 
\begin{itemize}
	\item $X$ is a curve, as a consequence of Thuillier's work~\cite{Thui} (see~\cite{GJKM});
	\item $K$ discretely or trivially valued, of residue characteristic $0$~\cite{siminag,trivval}, building on multiplier ideals and the Nadel vanishing theorem;
	\item $K$ is discretely valued of characteristic $p$, $(X,L)$ is defined over a function field of transcendence degree $d$, and resolution of singularities is assumed in dimension $d+n$~\cite{GJKM}, replacing multiplier ideals with test ideals.
\end{itemize}

Generalizing~\cite{BGJKM}, which dealt with the discretely valued case, the main result of~\cite[Theorem A]{BE} states that any two continuous metrics $\phi,\p$ on $L$ with continuous envelope satisfy
\begin{equation}\label{equ:BEint}
\vol(L,\phi,\p)=\en(\env(\phi),\env(\p)).
\end{equation}

In the present non-Archimedean context, the relative Monge--Amp\`ere energy can be interpreted as a local height, and~\eqref{equ:BEint} as a local Hilbert--Samuel formula. Combined with Theorem A, it enables us to prove the following analogue of~\cite[Theorem B]{BB}. 

\begin{thmB} \it{Assume that continuity of envelopes holds for $(X,L)$, and let $\phi$ be a continuous metric on $L$.} 
	\begin{itemize}
		\item[(i)] The Monge--Amp\`ere measure $(dd^c\env(\phi))^n$ is supported on the contact locus $\{\env(\phi)=\phi\}$. In other words, the \emph{orthogonality property} 
		$$
		\int_{X^\an}(\phi-\env(\phi))(dd^c\env(\phi))^n=0
		$$
		is satisfied.
		\item[(ii)] For any continuous function $f$ and continous psh metric $\p$, we have 
		$$
		\frac{d}{dt}\bigg|_{t=0}\en(\env(\phi+tf),\p)=\int_{X^\an}f\,(dd^c\env(\phi))^n.
		$$
	\end{itemize}
\end{thmB} 
It is in fact essentially formal to show that (i) and (ii) are equivalent, and are also equivalent to the special case of (ii) where $\phi$ is psh, which corresponds precisely to Theorem A, thanks to~\eqref{equ:BEint}. 

\medskip

Using Theorem B and the variational argument of~\cite{BBGZ,nama}, we are 
able to produce 'fundamental solutions' to Monge--Amp\`ere equations, as follows.  

\begin{corC} \it{Assume continuity of envelopes for $(X,L)$. Let $x\in X^\an$ be a nonpluripolar point, $\phi$ a continuous metric on $L$, and assume that $x$ is $L$-regular, in the sense that the envelope}
	$$
	\phi_x\coloneqq\sup\{\text{$\p$ psh metric on $L\mid\p(x)\le\phi(x)$}\}
	$$
	is continuous (and hence psh). Then 
	$$
	V^{-1}(dd^c\phi_x)^n=\d_x,
	$$
	with $V:=(L^n)$ and $\delta_x$ the Dirac mass at $x$. 
\end{corC} 
Here again, $L$-regularity is expected to be automatic for nonpluripolar points on a normal variety. It is established in~\cite[Theorem 5.13]{nakstab} when $X$ is smooth and $K$ is trivially or discretely valued, of residue characteristic $0$.

\medskip

As a final consequence of Theorem A, we generalize the equidistribution of Fekete points in Berkovich spaces, which was established in~\cite{BE} following the variational strategy going back to~\cite{BBW} in the complex analytic case, under assumptions guaranteeing the differentiability property (ii) of Corollary B. For any basis ${\bf s}=(s_1,\dots,s_N)$ of $H^0(X,L)$, the Vandermonde (or Slater) determinant 
$
\det(s_i(x_j))_{1\le i,j\le N} 
$
can be seen as a global section $\det({\bf s})\in H^0(X^N,L^{\boxtimes N})$. Given a continuous metric $\phi$ on $L$, a \emph{Fekete configuration} for $\phi$ is a point $P\in (X^N)^\an$ achieving the supremum of $|\det({\bf s})|_{\phi^{\boxtimes N}}$, a condition that does not depend on the choice of the basis ${\bf s}$. 
By Theorem A, the differentiability property \eqref{diff property} holds for any continuous psh metric $\phi$ of $L$ and hence we get the following result as a direct application of~\cite[Theorem 10.10]{BE}.

\begin{corD} \it{Let $K$ be any non-Archimedean field, and let $L$ be an ample line bundle on a projective, geometrically reduced $K$-scheme $X$. Set $n:=\dim X$, $N_m \coloneqq h^0(X,mL)$ and $V:=(L^n)$. Pick a continuous psh metric $\phi$ on $L$, and choose for each $m \gg 1$ a Fekete configuration $P_m\in (X^{N_m})^\an$ for $m\phi$. Then 
		$P_m$ equidistributes to the 
		probability measure $V^{-1}(dd^c\phi)^n$, \ie 
		$$
		\lim_{m \to \infty}\int_{X^\an} f\,\d_{P_m} =\int_{X^\an}f\,V^{-1}(dd^c\phi)^n. 
		$$
		for each continuous function $f$ on $X^\an$ where $\d_{P_m}$ is the discrete probability measure on $X^\an$ obtained by averaging over the components of the image of $P_m$ in $(X^\an)^{N_m}$. }
\end{corD}

%
%
%
%
\subsection*{Organization of paper}
Section 1 collects preliminary material on norms, metrics, and their relative volumes. We recall also properties of the energy and the Monge--Amp\`ere measures. 
Section 2 reviews some facts on the determinant of cohomology, and proves the key Riemann--Roch type formula.
In Section 3, we prove first Theorem A. Assuming continuity of envelopes, we  then deduce Corollary B and Corollary C.

\begin{ackn} We are  grateful to Jos\'e Burgos Gil, Antoine Ducros, Dennis Eriksson, Philipp Jell, Mattias Jonsson and Klaus K\"unnemann for many helpful discussions in relation to this work.
\end{ackn}

\addtocontents{toc}{\protect\setcounter{tocdepth}{0}}
\section*{Notation and conventions}
\addtocontents{toc}{\protect\setcounter{tocdepth}{1}}

Throughout the paper, we work over a \emph{non-Archimedean field} $K$, \ie a field complete with respect to a non-Archimedean absolute value $|\cdot|$, which might be the trivial absolute value. The corresponding valuation is denoted by $v_K\coloneqq-\log|\cdot|$. The valuation ring, maximal ideal and residue field are respectively denoted by 
$$
\Ko \coloneqq \{a \in K \mid |a|\le 1\},\quad\Koo \coloneqq  \{a \in K \mid |a|<1 \},\quad\Kt \coloneqq \Ko/\Koo.
$$ 
If $X$ is a scheme of finite type over $K$, we denote by $X^\an$ its Berkovich analytification. The space of continuous, real valued functions on $\Xan$ is denoted by $C^0(\Xan)$.

We use additive notation for line bundles and metrics. If $L,M$ are line bundles on $X$ endowed with metrics $\phi$ and $\psi$ respectively, then $L+M$ denotes the tensor product of the line bundles and $\phi+\psi$ the induced metric. The norm on $L$ associated to $\phi$ is denoted by $|\cdot|_\phi$ and $\|\cdot\|_\phi$ is the associated supnorm on $H^0(X,L)$, which is a norm if $X$ is reduced. See \S \ref{sec:metrics} for more details.

For line bundles $L_1, \dots, L_n$ on an $n$-dimensional projective scheme $X$ over a field, we use $(L_1\cdot \ldots\cdot L_n)$
for the intersection number of the first Chern classes of $L_1, \dots, L_n$. Usually, we will have $L=L_1= \dots = L_n$ and we then simply write $(L^n)$ for this intersection number, which agrees with the degree of $X$ with respect to $L$.


\section{Preliminaries}\label{sec:prelim}

We collect here some background results on the norms, lattices, models, Monge--Amp\`ere measures, energy and volumes. In what follows, $X$ denotes an $n$-dimensional, geometrically reduced\footnote{This simply amounts to $X$ reduced whenever $K$ is perfect.} projective $K$-scheme.
%
%
\subsection{Norms, lattices and content}\label{sec:norms}
Let $V$ be a finite dimensional $K$-vector space, and set $r=\dim V$. By a \emph{norm} on $V$, we always mean an ultrametric norm $\n:V\to\R_{\ge 0}$ compatible with the given absolute value of $K$. It induces a  {\it determinant norm} $\det\n$ on the determinant line $\det V= \Lambda^r V$, given by
$$
\det\| \tau \|  \coloneqq \inf_{\tau = v_1 \wedge \dots \wedge v_r} \|v_1\| \cdots \|v_r\|	
$$
for any $\tau \in \det V$. Given two norms $\n,\n'$, the {\it relative volume} of $\n$ with respect to  $\n'$ is defined as
$$
\vol(\n,\n')\coloneqq\log\left( \frac{\det\|\tau\|'}{\det\|\tau\|} \right)
$$
for any nonzero $\tau\in\det V$. For more details on the determinant norm and relative volumes, we refer to \cite[\S 2.1--2.3]{BE}.

A \emph{lattice} in $V$ is a finitely generated $\Ko$-submodule $\cV\subset V$ that span $V$ over $K$. The \emph{lattice norm} $\n_\cV$ associated to a lattice $\cV$ is given for $v \in V$ by
$$
\| v \|_\Vcal \coloneqq \min_{a \in K, \,  v \in a \Vcal}  |a|.
$$
Relative volumes of lattice norms admit the following algebraic interpretation. By \cite[Proposition 2.10 (i)]{Sch13} (see also~\cite[Lemma 2.17]{BE}), every finitely presented, torsion $\Ko$-module $M$ satisfies
$$ 
M \cong \Ko/( a_1) \oplus \ldots \oplus \Ko / ( a_r)
$$
for some nonzero $a_1,\dots,a_r\in\Koo$, where $r$ and the sequence $v(a_i)$ are further uniquely determined by $M$, up to reordering. The {\it content}\footnote{This quantity was called length in~\cite{Sch13}, and corresponds to $-\log$ of the content as defined in~\cite{Tem}.} of $M$ is defined as
$$
\length(M) = \sum_{i=1}^r v_K(a_i) \in \R_{\geq 0}.
$$
When $K$ is discretely valued with uniformizer $\pi\in\Koo$, then $\length(M)$ is the usual length of $M$, multiplied by $v_K(\pi)$~\cite[Example 2.19]{BE}. 

Now every finitely presented torsion $\Ko$-module $M$ arises as a quotient $M=\Vcal/\Vcal'$ for lattices $\Vcal'\subset\cV$ in a finite dimensional $K$-vector space, and  
\begin{equation}\label{equ:contvol}
\length(M)=\vol(\n_{\cV},\n_{\cV'}).
\end{equation}

%
%
\subsection{Metrics}\label{sec:metrics}
As in \cite[\S 5]{BE}, we use additive notation for metrics on a line bundle $L$ over $X$. Then a {\it metric} $\phi$ on $L$ is a family of functions $\phi_x:L \otimes_X \Hcal(x) \to \R \cup \{\infty\}$ such that $\abs_{\phi_x}  \coloneqq e^{-\phi_x}$ is a norm on the $1$-dimensional $\Hcal(x)$-vector space $L \otimes_X \Hcal(x)$ for every $x \in \Xan$. 
Here, $\Hcal(x)$ is the completed residue field of $x$ 
endowed with its canonical absolute value \cite[Remark 1.2.2]{berkovich-book}. We usually skip the $x$ and write simply $\abs_{\phi}$ for the norms. Note that $L \otimes_X \Hcal(x)$ is the non-Archimedean analogue of the fiber of a holomorphic line bundle.

Given two metrics $\phi$, $\psi$ on line bundles $L$, $M$ over $X$, we denote by $\phi\pm\p$  the induced metric on $L\pm M=L\otimes M^{\pm 1}$. The corresponding norms thus satisfy $|\cdot|_{\phi\pm\psi}= |\cdot|_\phi \otimes |\cdot|_\psi^{\pm 1}$.

A metric $\phi$ on $L$ is called {\it continuous} if the function $x \mapsto |t(x)|_\phi$, induced by any local section $t$ of $L$, is continuous with respect to the Berkovich topology.
For $s\in H^0(X,L)$, 	the associated  {\it supremum norm}  is denoted by 
$$
\|s\|_\phi \coloneqq \sup_{x \in X^\an}|s(x)|_\phi.
$$

%
\subsection{Models}\label{sec:models}
In this paper, a {\it model} $\Xcal$ of $X$ is a flat projective $K^\circ$-scheme, together with an identification of the generic fiber $\Xcal_\eta$ of $\Xcal\to\Spec(K^\circ)$ with $X$. There is a canonical {\it reduction map} $\redu_\Xcal:\Xan \to \Xcal_s$ to the special fiber $\Xcal_s$ of $\Xcal$  (see \cite[Remark 2.3]{GM} and \cite[\S 2]{gubler-rabinoff-werner2} for details).

We say that a model $\Xcal$ of $X$ is {\it dominated} by another model $\Xcal'$ if the identity on $X$ extends to a (unique) morphism $\Xcal' \to \Xcal$ over $\Ko$. This induces a partial order on the set of models of $X$ modulo isomorphim, which turns it into a directed system. 

If $K$ is algebraically closed and nontrivially valued, then it follows from the reduced fiber theorem (see for instance~\cite[Theorem 4.20]{BE}) that models $\Xcal$ with reduced special fiber $\Xcal_s$ are cofinal among all models. On the other hand, in the trivially valued case, $X$ is its only model, up to isomorphism.  

Now let $L$ be a line bundle on $X$. A {\it model $(\Xcal,\Lcal)$ of $(X,L)$} consists of a model $\Xcal$ of $X$ and a line bundle $\Lcal$ on $\Xcal$ together with an identification  $\Lcal|_{\Xcal_\eta} \simeq L$ compatible with the identification $\Xcal_\eta \simeq X$. We then say that $\Lcal$ is a {\it model of $L$ determined on $\Xcal$}. Every model of the trivial line bundle $L=\cO_X$ determined on a model $\Xcal$ is of the form $\Lcal=\cO_\Xcal(D)$, where $D$ is a Cartier divisor which is \emph{vertical}, \ie supported in the special fiber. 

\begin{lem}\label{lem:ample} Assume $K$ is algebraically closed and non-trivially valued, and let $(L_i)$ be a finite collection of ample line bundles on $X$. Then models $\Xcal$ of $X$ that have reduced special fiber and such that all $L_i$ extend to an ample $\Q$-line bundles on $\Xcal$ are cofinal in the set of all models. 
\end{lem}
\begin{proof} By  \cite[Proposition 4.11, Lemma 4.12]{GM}, every model $\Xcal$ of $X$ is dominated by a model $\Xcal'$ on which all $L_i$ extend to ample $\Q$-line bundles $\cL'_i$. By~\cite[Theorem 4.20]{BE}, the integral closure of $\Xcal'$ in its generic fiber $\Xcal_\eta\simeq X$ is a model $\Xcal''$ with reduced special fiber, which dominates $\Xcal'$ via a finite morphism $\mu:\Xcal''\to\Xcal'$. As a result, $\mu^*\Lcal'_i$ is an ample $\Q$-line bundle extending $L_i$, and we are done.  
\end{proof}

If $(\Xcal,\Lcal)$ is a model of $(X,L)$, then $H^0(\Xcal,\Lcal)$ is a lattice in $H^0(X,L)$. Indeed, it follows from the direct image theorem given in \cite[Theorem 3.5]{Ull} that $H^0(\Xcal,\Lcal)$ is a finitely generated $\Ko$-module, while flat base change implies $H^0(\Xcal,\Lcal)\otimes_\Ko K\simeq H^0(X,L)$. 

Recall that a section $s$ of a line bundle over a scheme $Z$ is \emph{regular} if its zero subscheme is a Cartier divisor, \ie if the corresponding function in any local trivialization of the line bundle is a nonzerodivisor. The section $s$ is \emph{relatively regular} with respect to a flat morphism $Z\to S$ if its zero subscheme is a Cartier divisor and is flat over $S$. 

Given a model $(\Xcal,\Lcal)$ of $(X,L)$, it follows from~\cite[11.3.7]{EGAIV4} that a section $s\in H^0(\Xcal,\Lcal)$ is relatively regular (with respect to the structure morphism $\cX\to\Spec K^\circ$) if and only if its restriction to $\Xcal_s$ is regular. By~\cite[Proposition A.15]{BE}, if $\Lcal$ is ample then $H^0(\Xcal,m\Lcal)$ admits relatively regular sections for all $m\gg 1$. For later use, we note: 

\begin{lemma}\label{lem:relreg}
	Let $(\Xcal,\Lcal)$ be a model of $(X,L)$, and $D$ be an effective vertical Cartier divisor. If $s\in H^0(\Xcal,\Lcal)$ is a relatively regular section, then $s|_D$ is regular on $D$.
\end{lemma}
\begin{proof} The statement is local, and thus reduces to the following. Let $\Acal$ be a flat, finite type $\Ko$-algebra, $f\in\Acal$ a relatively regular function, and $a\in\Acal$ a nonzerodivisor whose image in $\Acal\otimes_\Ko K$ is invertible. We have to show that the image of $f$ in $\Acal/(a)$ is a nonzerodivisor. To see this, pick $g,h\in\Acal$ such that $fg=ah$. We then need to prove that $g\in(a)$. Since $f$ is relatively regular, $\Acal/(f)$ is flat over $\Ko$, and the map $\Acal/(f) \to\Acal/(f) \otimes_\Ko K$ is thus injective. The image of $a$ in $\cA/(f)\otimes_\Ko K$ being invertible, the image of $h$ in $\Acal/(f)\hookrightarrow\cA/(f)\otimes_\Ko K$ is zero, and hence $h\in(f)$, \ie $h=h'f$ for some $h' \in \cA$. Then $fg=ah'f$, and hence $g=ah'\in (a)$ as $f$ is a nonzerodivisor. 
\end{proof}

%
%
\subsection{Model metrics}
Let $L$ be a line bundle on $X$. To every model $(\Xcal,\Lcal)$ of $(X,L)$ is associated a continuous metric $\phi_\Lcal$ on $L$, determined as follows:  every point of $\Xan$ belongs to the affinoid domain $\redu_{\Xcal}^{-1}(\cU)$ induced by an affine open subset $\cU$ of $\Xcal$ on which $\Lcal$ admits a trivializing section $\tau$, and $\phi_\Lcal$ is determined by requiring that $|\tau|_{\phi_\Lcal}\equiv 1$ on $\redu_{\Xcal}^{-1}(\cU)$. This construction is invariant under pull-back to a higher model, \ie $\phi_{\mu^*\Lcal}=\phi_\Lcal$ for any morphism of models $\mu:\Xcal'\to\Xcal$. We refer to \cite[5.3]{BE} and \cite[\S 2]{GM} for more details. 

A \emph{model metric} on $L$ is defined as a continuous metric of the form $\phi=m^{-1}\phi_\Lcal$ where $\Lcal$ is a model of $mL$ for some nonzero $m \in \N$. We  say that $\phi$ is determined by the $\Q$-model $m^{-1}\Lcal$. 

A \emph{model function} is a continuous function on $X^\an$ corresponding to a model metric on the trivial line bundle $\cO_X$. It is thus determined by a vertical $\Q$-Cartier divisor $D$ on some model $\Xcal$ of $X$, and we write $\phi_D$ for the corresponding model function. Model functions form a $\Q$-vector space of continuous functions, which is stable under max. When $K$ is non-trivially valued, model functions further separate points, and hence are dense in $C^0(X^\an)$ by the Stone-Weierstrass theorem, see~\cite[Theorem 7.12]{gubler-crelle}. 

The next result explains the importance of models with reduced special fiber in our approach. 

\begin{lemma}\label{lem:intclosnorm} Let $\Xcal$ be a model of $X$ with reduced special fiber. 
	\begin{itemize}
		\item[(i)] If $\Lcal$ is a model of $L$ determined on $\Xcal$, then the supnorm $\n_{\phi_\Lcal}$ coincides with the lattice norm $\n_{H^0(\Xcal,\Lcal)}$.
		\item[(ii)] If $D$ is a vertical Cartier divisor on $\Xcal$, then $D$ is effective if and only if $\phi_D\ge 0$. 
	\end{itemize}
\end{lemma} 

\begin{proof} Property (i) is~\cite[Lemma 6.3]{BE}. For (ii), note that the vertical Cartier divisor $D$ induces a canonical meromorphic section $s_D$ of $\Lcal=\Ocal(D)$ which restricts to a global section of $L=\Lcal|_X$. By definition  of a lattice norm, we have $s_D \in H^0(\Xcal,\Lcal)$ if and only if $\|s_D\|_{H^0(\Xcal,\Lcal)} \leq 1$ and hence (ii) follows from (i).
\end{proof}

%
%

\subsection{Plurisubharmonic metrics and envelopes}\label{sec:psh}
In this subsection, we recall some facts about plurisubharmonic metrics on an ample line bundle $L$ over $X$. We refer to~\cite[\S 7]{BE} for a thorough discussion. 

Assume first that $K$ is non-trivially valued. Following Shou-Wu Zhang~\cite{zhang-95}, we then say that a continuous metric $\phi$ on $L$ is \emph{plurisubharmonic} (\emph{psh} for short) if $\phi$ can be written as a uniform limit of model metrics $\phi_{\Lcal_i}$ associated to nef $\Q$-models $\Lcal_i$ of $L$. By~\cite[Theorem 7.8]{BE}, this definition is compatible with the point of view of~\cite{BE,trivval}, which defines continuous psh metrics as uniform limits of Fubini--Study metrics. 

When $K$ is trivially valued, a continuous metric $\phi$ on $L$ is called \emph{psh} if there exists a non-trivially valued non-Archimedean field extension $F$ of $K$ such that the induced continuous metric $\phi_F$ on the base change $L\otimes_KF$ is psh in the above sense. By~\cite[Theorem 7.32]{BE}, this condition is independent of the choice of $F$, and compatible with the Fubini--Study approach of~\cite{BE,trivval}. 

\begin{defi} We say that \emph{continuity of envelopes} holds for $(X,L)$ if, for any continuous metric $\phi$ on $L$, the \emph{psh envelope} 
	$$
	\env(\phi)\coloneqq\sup\{\p\text{ continuous psh metric on }L\mid\p\le\phi\}
	$$
	is a continuous metric on $L$ as well. 
\end{defi}
When this holds, $\env(\phi)$ is automatically psh, and is thus characterized as the greatest continuous psh metric dominated by $\phi$. In the complex analytic case, continuity of envelopes holds over any normal complex space, and fails in general otherwise. By analogy, we conjecture that continuity of envelopes holds as soon as $X$ is normal. As recalled in the introduction, it is at present known when $X$ is smooth and one of the following is satisfied:
\begin{itemize}
	\item $X$ is a curve, as a consequence of A.~Thuillier's work~\cite{Thui} (see~\cite{GJKM});
	\item $K$ discretely or trivially valued, of residue characteristic $0$~\cite{siminag,trivval}, building on multiplier ideals and the Nadel vanishing theorem;
	\item $K$ is discretely valued of characteristic $p$, $(X,L)$ is defined over a function field of transcendence degree $d$, and resolution of singularities is assumed in dimension $d+n$~\cite{GJKM}, replacing multiplier ideals with test ideals.
\end{itemize}
%
\subsection{Monge--Amp\`ere measures and energy}\label{subsection MA measures}
A construction of A.~Chambert-Loir associates to any $n$-tuple $\phi_1,\dots,\phi_n$ of continuous psh metrics  on ample line bundles $L_1, \dots, L_n$ over $X$ their \emph{mixed Monge--Amp\`ere measure}
$$
dd^c\phi_1\wedge\dots\wedge dd^c\phi_n,
$$
a positive Radon measure on $X^\an$ of total mass equal to the intersection number $(L_1\cdot\ldots\cdot L_n)$. This measure depends multilinearly and continuously on the tuple $(\phi_1,\dots,\phi_n)$ with respect to uniform convergence (and weak convergence of measures), and the construction is further compatible with ground field extension.

These measures were first constructed in~\cite{CL} over non-Archimedean fields $K$ with a countable dense subset. Over an arbitrary non-Archimedean ground field, the measures can be obtained by base change to a non-trivially valued algebraically closed non-Archimedean field $F$, using~\cite[\S 2]{gubler-2007b}. One can also directly rely on the local approach in~\cite{CL}, see~\cite[\S 8.1]{BE} for details.

\begin{example} \label{model example for MA measure} For psh model metrics $\phi_1,\dots,\phi_n$, the measure  $dd^c\phi_1\wedge\dots\wedge dd^c\phi_n$ has finite support. When $K$ is algebraically closed, the $\phi_i$ are determined by nef $\Q$-models $\Lcal_1,\dots,\Lcal_n$ of $L$ determined on a model $\Xcal$ that can be chosen to have reduced special fiber $\Xcal_s$; each irreducible component $Y$ of $\Xcal_s$ then determines a unique point $x_Y\in X^\an$ with $\redu_\Xcal(x_Y)$ the generic point of $Y$, and we have 
	$$
	dd^c \phi_1 \wedge \dots \wedge dd^c \phi_n = \sum_Y (\Lcal_1|_Y\cdots\Lcal_n|_Y)\delta_{x_Y}, 
	$$
	where $\delta_{x_Y}$ is the Dirac measure at $x_Y$, see~\cite[Corollary 2.8]{gubler-2007b},~\cite[Th\'eor\`eme 6.9.3]{ChaDuc}.
\end{example}

From now on we fix an ample line bundle $L$ on $X$, and denote by $V\coloneqq(L^n)$ its volume. The \emph{relative Monge--Amp\`ere energy} of $\phi,\p$ is defined as\footnote{We emphasize that the present normalization is not uniform across the literature.} 
\begin{equation}\label{equ:en}
\en(\phi,\p):=\frac{1}{n+1}\sum_{j=0}^n\int_{X^\an} (\phi-\p)(dd^c\phi)^j\wedge(dd^c\p)^{n-j}.
\end{equation}
For each $\p$, the functional $\phi\mapsto\en(\phi,\p)$ is characterized as the unique antiderivative of the Monge--Amp\`ere operator $\phi\mapsto (dd^c\phi)^n$ that vanishes at $\p$, in the sense that
\begin{equation}\label{equ:endiff}
\frac{d}{dt}\bigg|_{t=0}\en((1-t)\phi+t\phi',\p)=\int_{X^\an}(\phi'-\phi)(dd^c\phi)^n
\end{equation}
for any two continuous psh metrics $\phi,\phi'$. As a consequence, the \emph{cocycle property}
$$
\en(\phi_1,\phi_2)+\en(\phi_2,\phi_3)+\en(\phi_3,\phi_1)=0
$$
holds for all triples of continuous psh metrics $\phi_1,\phi_2,\phi_3$ on $L$. 

Another key property of the Monge--Amp\`ere energy is the concavity of $\phi\mapsto\en(\phi,\p)$. In view of~\eqref{equ:endiff} and the cocyle property, this amounts to
\begin{equation}\label{equ:Econc}
\en(\phi,\p)\le\int_{X^\an}(\phi-\p)(dd^c\p)^n
\end{equation}
for all continuous psh metrics $\phi,\p$ on $L$. Moreover, 
$$
\en(\phi+c)=\en(\phi)+Vc
$$
for all $c\in\R$. We refer to~\cite[\S 3.8]{trivval} for details on the above properties. 
%
%
\subsection{Relative volumes of metrics}\label{subsection comparison volumes}
Recall that the \emph{volume} of a line bundle $L$ on $X$ is defined as 
$$
\vol(L):=\lim_{m\to\infty}\frac{n!}{m^n}\dim H^0(X,mL)\in\R_{\ge 0}
$$
(see for instance~\cite[Theorem 9.8]{BE} for the existence of the limit in the present generality). We have $\vol(L)>0$ if and only if $L$ is big, and $\vol(L)=(L^n)$ whenever $L$ is nef. 

The \emph{relative volume} of two continuous metrics $\phi,\p$ on $L$ is
$$
\vol(L,\phi,\p)\coloneqq\lim_{m\to\infty}\frac{n!}{m^{n+1}}\vol(\n_{m\phi},\n_{m\p})\in\R. 
$$
The existence of this limit was established is~\cite[Theorem 9.8]{BE}, building on the work of Chen and Maclean~\cite{CMac}.

\begin{prop}\label{prop:volmetr} The following properties hold for all continuous metrics on a given line bundle $L$: 
	\begin{itemize}
		\item[(i)] \emph{cocycle formula:} $\vol(L,\phi_1,\phi_2)+\vol(L,\phi_2,\phi_3)+\vol(L,\phi_3,\phi_1)=0$; 
		\item[(ii)] \emph{monotonicity:} $\phi\le\phi'\Longrightarrow\vol(L,\phi,\p)\le\vol(L,\phi',\p)$; 
		\item[(iii)] \emph{scaling:} $\vol(L,\phi+c,\p)=\vol(L,\phi,\p)+\vol(L)c$ for $c\in\R$; 
		\item[(iv)] \emph{Lipschitz continuity:}
		$$
		\left|\vol(L,\phi,\p)-\vol(L,\phi',\p')\right|\le\vol(L)\left(\sup_{x \in \Xan}|\phi(x)-\phi'(x)|+\sup_{x \in \Xan}|\p(x)-\p'(x)|\right). 
		$$
		\item[(v)] \emph{homogeneity}: $\vol(aL,a\phi,a\p)=a^{n+1}\vol(L,\phi,\p)$ for all $a\in\N$. 
		\item[(vi)] \emph{base change invariance}: for any non-Archimedean extension $F/K$, we have
		$$\vol(L_F,\phi_F,\p_F)=\vol(L,\phi,\p),$$
		with $\phi_F,\p_F$ denoting the pullbacks of $\phi,\p$ to the base change $L_F=L \otimes_K F$. 
		
	\end{itemize}
\end{prop}
In particular, if $L$ is not big, \ie $\vol(L)=0$, then $\vol(L,\phi,\p)=0$ for all continuous metrics on $L$, by (iv). The next result, which equates relative volume and relative energy, goes back to~\cite{BB} in the complex analytic case. In the non-Archimedean context, the result was established in~\cite{BGJKM} in the discretely valued case, and in~\cite[Corollary B]{BE} in the general case. 

\begin{thm}\label{thm:BE} If $L$ is an ample line bundle and $\phi,\p$ are psh metrics on $L$ with continuous psh envelopes $\env(\phi),\env(\p)$, then 
	$$
	\vol(L,\phi,\p)=\en(\env(\phi),\env(\p)). 
	$$
\end{thm}

%
%


\section{An asymptotic Riemann--Roch theorem}\label{sec:RR}
This section reviews some facts on the determinant of cohomology and Deligne pairings, following~\cite[Appendix A]{BE}, and uses this to prove a Riemann--Roch type formula for vertical Cartier divisors on models. We still denote by $X$ a geometrically reduced projective $K$-scheme of dimension $n$. 
%
%
\subsection{Determinant of cohomology and Deligne pairings}

The determinant of cohomology of a line bundle $L$ on $X$ is a line bundle $\la_X(L)$ over $\Spec K$, \ie a one-dimensional $K$-vector space; it can simply be described as
$$
\la_X(L):=\sum_{i=0}^n(-1)^i\det H^i(X,L),
$$
where we use additive notation for tensor products of line bundles. 

Consider now a model $(\Xcal,\Lcal)$ of $(X,L)$, with structure morphism $\pi:\Xcal\to S:=\Spec K^\circ$. Kiehl's theorem on (pseudo)coherence of direct images and the flatness of $\pi$ imply that the complex $R\pi_*\Lcal$ is perfect. Thus there exists a bounded complex of vector bundles $\cE^\bullet$ on $S$ with a quasi-isomorphism $\cE^\bullet\to R\pi_*\Lcal$ and the determinant of cohomology of $\cL$ is defined as 
$$
\la_\Xcal(\Lcal) \coloneqq \det\cE^\bullet=\sum_i(-1)^i\det\cE^i,
$$
this line bundle on $S$ being unique up to unique isomorphism of $\Q$-line bundles by~\cite{KM}. This construction commutes with base change, and $\la_\Xcal(\Lcal)$ is thus a $\Q$-model of $\la_X(L)$, cf.~\cite[Appendix A]{BE} for more details. 

By flatness of $\pi$, the  $\cO_S$-module $\pi_*\Lcal$ is torsion-free, and hence locally free. When $R^i\pi_*\Lcal$ is locally free for all $i$, \cite[p.43]{KM} yields 
\begin{equation} \label{concrete formula for det}
\la_\Xcal(\Lcal)=\sum_{i=0}^n(-1)^i\det R^i\pi_*\Lcal.
\end{equation}
Combining with Serre vanishing (see~\cite[Corollary A.12]{BE} for the relevant statement), we infer: 

\begin{lem}\label{lem:detcohom} If $\Lcal$ is ample and $\Ecal$ is any line bundle on $\Xcal$, then $\la_\Xcal(m\Lcal+\Ecal)$ coincides with the determinant of the vector bundle $\pi_*(m\Lcal+\Ecal)$ for all $m\gg 1$. 
\end{lem}

The fundamental property of the determinant of cohomology, which is extracted in~\cite[Appendix A]{BE} from a paper of F.~Ducrot~\cite{Ducrot}, is that $\la_\Xcal$ admits a canonical structure of a \emph{polynomial functor of degree $n+1$}. By definition, this means that the $(n+1)$-st iterated difference 
\begin{equation}\label{equ:deligne}
\langle\Lcal_0,\dots,\Lcal_n\rangle_\Xcal:=\sum_{I\subset\{0,\dots,n\}}(-1)^{n+1-|I|}\la_\Xcal(\sum_{i\in I}\Lcal_i)
\end{equation}
has a structure of multilinear functor, compatible with its natural symmetry structure and with base change, and called the \emph{Deligne pairing}. As a consequence, we get for each line bundle $\Lcal$ on a model $\Xcal$ and $m\in\Z$ a polynomial expansion of $\Q$-line bundles
\begin{equation}\label{equ:KM}
\la_\Xcal(m\Lcal)=\frac{m^{n+1}}{(n+1)!}\langle\cL^{n+1}\rangle_\Xcal+ \dots,
\end{equation}
called the \emph{Knudsen--Mumford expansion}. Here and thereafter, we use the shorthand notation
$$
\langle\cL^{n+1}\rangle_\Xcal \coloneqq \langle  \underbrace{\cL, \dots, \cL}_{\text{$n+1$-times}} \rangle_\Xcal.
$$

\begin{lem}\label{lem:firstdiff} Let $\Lcal_0$ be a be a line bundle on a model $\Xcal$ of $X$. The polynomial structure of degree $n+1$ on $\la_\Xcal$ induces a polynomial structure of degree $n$ on 
	$$
	\Lcal\mapsto\la_\Xcal(\Lcal+\Lcal_0)-\la_\Xcal(\Lcal),
	$$
	whose $n$-th iterated difference further identifies with 
	$$
	(\Lcal_1,\dots,\Lcal_n)\mapsto\langle\Lcal_0,\Lcal_1,\dots,\Lcal_n\rangle_\Xcal.
	$$
\end{lem}
\begin{proof} By definition, the $n$-th iterated difference of $\Lcal\mapsto\la_\Xcal(\Lcal+\Lcal_0)-\la_\Xcal(\Lcal)$ is equal to 
	$$
	\sum_{J\subset\{1,\dots,n\}}(-1)^{n-|J|}\left(\la_\Xcal(\Lcal_0+\sum_{j\in J}\Lcal_j)-\la_\Xcal(\sum_{j\in J}\Lcal_j)\right)
	$$
	$$
	=\sum_{I\subset\{0,\dots,n\},\,0\in I}(-1)^{n+1-|I|}\la_\Xcal(\sum_{i\in I}\Lcal_i)+\sum_{I\subset\{0,\dots,n\},\,0\notin I}(-1)^{n+1-|I|}\la_\Xcal(\sum_{i\in I}\Lcal_i)
	$$
	$$
	=\langle\Lcal_0,\Lcal_1,\dots,\Lcal_n\rangle_\Xcal,
	$$
	by~\eqref{equ:deligne}. This finishes the proof, since the latter is a multilinear functor of $(\Lcal_1,\dots,\Lcal_n)$.
\end{proof}

We finally recall the following special case of~\cite[Theorem 8.18]{BE}. 

\begin{lem}\label{lem:delint} If $D$ is a vertical divisor on a model $\Xcal$ of $X$ with associated model function $\phi_D$ and $\Lcal_1,\dots,\Lcal_n$ are nef line bundles on $\Xcal$, then 
	\begin{equation*}\label{equ:deligne_int}
	\phi_{\langle\cO_\Xcal(D),\Lcal_1,\dots,\Lcal_n\rangle_\Xcal}=\int_{X^\an}\phi_D\,dd^c\phi_{\Lcal_1}\wedge\dots\wedge dd^c\phi_{\Lcal_n},
	\end{equation*}
	where we identify the model function $\phi_{\langle\cO_\Xcal(D),\Lcal_1,\dots,\Lcal_n\rangle_\Xcal}$ on $\Spec(K)$ with its unique value.
\end{lem}

%
%
\subsection{An asymptotic Riemann--Roch theorem}
Pick a model of $X$, a line bundle $\Lcal$ on $\Xcal$, and an effective vertical Cartier divisor $D$ on $\Xcal$. By the coherence in Kiehl's direct image theorem, the $K^\circ$-module $H^0(D,\Lcal|_D)$ is finitely presented and torsion (see~\cite[Corollary A.12]{BE}). We denote by $h^0(D,\Lcal|_D)$ its content, as defined in~\S\ref{sec:norms}.

\begin{thm}\label{thm:key} Let $\Xcal$ be a model of $X$, $\Lcal$ an ample line bundle on $\Xcal$, and $D$ an effective vertical Cartier divisor on $\Xcal$. Then 
	$$
	h^0(D,m\Lcal|_D)=\frac{m^n}{n!}\int_{X^\an} \phi_D(dd^c\phi_\Lcal)^n+O(m^{n-1}). 
	$$
\end{thm}

\begin{proof} By Serre vanishing~\cite[Theorem A.6]{BE}), we have $H^q(\Xcal,m\Lcal)=H^q(\Xcal,m\Lcal-D)=0$ for all $q\ge 1$ and $m\gg 1$. Restriction to $D$ thus yields an exact sequence
	$$
	0\to H^0(\Xcal,m\Lcal-D)\to H^0(\Xcal,m\Lcal)\to H^0(D,m\Lcal|_D)\to 0,
	$$
	which implies
	$$
	h^0(D,m\Lcal|_D)=\phi_{\det H^0(\Xcal,m\Lcal)}-\phi_{\det H^0(\Xcal,m\Lcal-D)},
	$$
	by~\eqref{equ:contvol}. By Lemma~\ref{lem:detcohom}, we further have 
	$$
	\det H^0(\Xcal,m\Lcal)=\la_\Xcal(m\Lcal),\quad \det H^0(\Xcal,m\Lcal-D)=\la_\Xcal(m\Lcal-D)
	$$
	and hence 
	$$
	h^0(D,m\Lcal|_D)=\phi_{\la_\Xcal(m\Lcal)-\la_\Xcal(m\Lcal-D)}.
	$$
	Now Lemma~\ref{lem:firstdiff} provides a polynomial expansion
	$$
	\la_\Xcal(m\Lcal)-\la_\Xcal(m\Lcal-D)=\frac{m^n}{n!}\langle\cO_\Xcal(D),\Lcal^n\rangle_\Xcal+...,
	$$
	and hence 
	$$
	h^0(D,m\Lcal|_D)=\frac{m^n}{n!}\phi_{\langle\cO_\Xcal(D),\Lcal^n\rangle_\Xcal}+O(m^{n-1})
	$$
	$$
	=\frac{m^n}{n!}\int_{X^\an} \phi_D (dd^c\phi_\Lcal)^n+O(m^{n-1}),
	$$
	by Lemma~\ref{lem:delint}. 
\end{proof}


\section{Differentiability and orthogonality}\label{section inequality}

In this section, we prove our main result on differentiability of relative volumes, which generalizes \cite[Theorem B]{BGJKM} from discretely valued non-Archimedean fields to arbitrary ones. In what follows, $X$ is a projective, geometrically reduced scheme of dimension $n$ over an arbitrary non-Archimedean field $K$, and $L$ is an \emph{ample} line bundle on $X$. 

%
%
\subsection{Proof of Theorem A}

The following result corresponds to Theorem A in the introduction. 

\begin{thm}\label{thm:main} For any continuous psh metric $\phi$ on $L$ and continuous function $f$ on $X^\an$, we have 
	$$
	\frac{d}{dt}\bigg|_{t=0}\vol(L,\phi+tf,\phi)= \int_{X^\an} f\,(dd^c\phi)^n.
	$$
\end{thm}
The key ingredient in the proof is the following general estimate, which can be viewed as a local analogue of the Siu-type inequality proved in~\cite{Yuan}. 

\begin{lem}\label{lem:main}
	Let $\phi$ be a continuous psh metric on $L$, $\psi_1,\psi_2$ be continuous psh metrics on an auxiliary ample line bundle $M$, and set $f:=\psi_1-\psi_2$ and $C:=((L+M)^n)-(L^n)>0$. Then 
	\begin{equation}\label{equ:volMA}
	C\inf_{x \in X^\an}f(x)\le\int_\Xan f \, ( dd^c \phi + dd^c \psi_1)^n-\vol(L,\phi+f,\phi)\le C\sup_{x \in X^\an} f(x). 
	\end{equation}
\end{lem}

\begin{proof} In the proof, we assume that the reader is familiar with the properties of Monge--Amp\`ere measures and relative volumes given in \S \ref{subsection MA measures} and in \S \ref{subsection comparison volumes}.
	First, we give a few reduction steps. 
	
	By the invariance of relative volumes under ground field extension, 
	we can pass to a non-Archimedean extension and assume that $K$ is algebraically closed and non-trivially valued. Every continuous psh metric on an ample line bundle is then a uniform limit of metrics induced by nef $\Q$-models of $L$. By continuity of Monge--Amp\`ere measures and relative volumes with respect to uniform convergence, we may thus assume that there exist nef $\Q$-models $\Lcal$ and $\Mcal_i$ of $L$ and $M$, determined on a model $\Xcal$ of $X$, such that $\phi=\phi_\Lcal$ and $\psi_i=\phi_{\Mcal_i}$. Since $K$ is algebraically closed, we can further assume after passing to a higher model that $\Xcal$ has reduced special fiber, and that $L$ and $M$ admit ample $\Q$-models $\Lcal'$, $\Mcal'$ on $\Xcal$, by Lemma~\ref{lem:ample}.  
	Replacing $\Lcal$ and $\Mcal_i$ with $(1-\e)\Lcal+\e\Lcal'$ and $(1-\e)\Mcal_i+\e\Mcal'$, $0<\e\ll 1$, we are thus reduced to the case where $\Lcal$ and the $\Mcal_i$ themselves are ample $\Q$-line bundles, using again the continuity of Monge--Amp\`ere measures and relative volumes with respect to uniform convergence. Replacing $L$ and $M$ with large enough multiples and using the homogeneity property of relative volumes, 
	we can finally assume that $\Lcal$ and the $\Mcal_i$ are honest ample line bundles on $\Xcal$ such that each admits a relatively regular section, using~\cite[Proposition A.15]{BE}. 
	
	Observe that adding to $f$ a constant $a\in\R$ translates the quantity
	$$
	\int_\Xan f \, ( dd^c \phi + dd^c \psi_1)^n-\vol(L,\phi+f,\phi)
	$$ 
	by $aC$. In order to prove the left-hand inequality in~\eqref{equ:volMA}, we may thus replace $f$ with $f-\inf_{X^\an} f$ and assume $\inf_{X^\an} f= 0$. The unique vertical Cartier divisor $E$ on $\Xcal$ such that $\Mcal_1-\Mcal_2=\cO(E)$ satisfies $\phi_E=f \geq 0$, and $E$ is thus effective by Lemma~\ref{lem:intclosnorm}, since $\Xcal$ has reduced special fiber. Pick integers $1\le j\le m$. The restriction exact sequence 
	$$
	0\to H^0\left(\Xcal,m\Lcal+(j-1) E\right)\to H^0\left(\Xcal,m\Lcal+jE\right))\to H^0\left(E,(m\Lcal+jE)|_E\right)
	$$
	yields 
	$$
	\vol\left(\n_{m\phi+j f},\n_{m\phi+(j-1)f}\right)=\vol\left(\n_{H^0\left(\Xcal,m\Lcal+jE\right)},\n_{H^0\left(\Xcal,m\Lcal+(j-1)E\right)}\right)
	$$
	$$
	\le h^0\left(E,(m\Lcal+jE)|_E\right),  
	$$
	where the first equality follows from Lemma~\ref{lem:intclosnorm} and the inequality follows from~\eqref{equ:contvol}. Summing up over $j$ and using the cocycle property of relative volumes, we infer
	$$
	\vol\left(\n_{m(\phi+f)},\n_{m\phi}\right)\le \sum_{j=1}^m h^0\left(E,(m\Lcal+jE)|_E\right). 
	$$
	Since $\Mcal_1$ and $\Mcal_2$ admit relatively regular sections, their restrictions to $E$ admit regular sections as well, by Lemma~\ref{lem:relreg}. For $j=1,\dots,m$ we thus have 
	$$
	h^0\left(E,(m\Lcal+jE)|_E\right)=h^0\left(E,(m\Lcal+j\Mcal_1-j\Mcal_2)|_E\right)\le h^0\left(E,m(\Lcal+\Mcal_1)|_E\right), 
	$$
	and hence
	$$
	\vol\left(\n_{m(\phi+f)},\n_{m\phi}\right)\le m\,h^0\left(E,m(\Lcal+\Mcal_1)|_E\right).
	$$
	As a result, 
	$$
	\vol\left(L,\phi+f,\phi\right)=\lim_{m \to \infty} \frac{n!}{m^{n+1}}\vol\left(\n_{m(\phi+f)},\n_{m\phi}\right)
	$$
	$$
	\le\lim_{m \to \infty}\frac{n!}{m^n}h^0\left(E,m(\Lcal+\Mcal_1)|_E\right)=\int_\Xan f \, ( dd^c \phi + dd^c \psi_1)^{ n},
	$$
	where the last equality follows from Theorem~\ref{thm:key}. This concludes the proof of the left-hand inequality in~\eqref{equ:volMA}. 
	
	The proof of the right-hand inequality is very similar. In that case, we may replace $f$ with $f-\sup_{X^\an} f$ and assume $\sup_{X^\an} f= 0$. As a result, the vertical Cartier divisor $D$ with $\Ocal(D)=\cM_2-\cM_1$ is effective, using  $\phi_D=-f \geq 0$. The restriction exact sequence 
	$$
	0\to H^0\left(\Xcal,m\Lcal-(j+1)D\right)\to H^0\left(\Xcal,m\Lcal-jD\right)\to H^0\left(D,(m\Lcal-jD)|_D\right)
	$$
	then shows that
	$$
	\vol\left(\n_{m\phi},\n_{m(\phi+f)}\right)\le \sum_{j=0}^{m-1}h^0\left(D,(m\Lcal-jD)|_D\right)\le m\, h^0\left(D,m(\Lcal+\Mcal_1)|_D\right),
	$$
	which yields
	$$
	-\vol\left(L,\phi+f,\phi\right)=\vol\left(L,\phi,\phi+f\right)
	$$
	$$
	\le \int_{X^\an}\phi_D\left(dd^c(\phi+\psi_1)\right)^n=-\int_{X^\an} f\left(dd^c(\phi+\psi_1)\right)^n 
	$$
	proving the right-hand inequality and hence the claim.
\end{proof}

\begin{proof}[Proof of Theorem~\ref{thm:main}] Let $\phi$ be a continuous psh metric on $L$ and $f$ be a continuous function on $X^\an$. Assume first that there exist continuous psh metrics $\psi_1,\psi_2$ on an ample line bundle $M$ such that $f=\psi_1-\psi_2$. Pick $m\in\Z_{>0}$, $t\in(0,m^{-1}] $, and observe that $mtf=\psi_1-\psi_{2,t}$ where 
	$$
	\psi_{2,t}:=\psi_1-mt f=(1-mt)\psi_1+mt\psi_2
	$$
	is a continuous psh metric on $M$, as a convex combination of such metrics. By Lemma~\ref{lem:main}, we thus have
	$$
	tm C_m\inf_{x \in X^\an} f(x)\le mt\int_{X^\an} f\,(m dd^c\phi+dd^c\psi_1)^n-\vol(mL,m\phi+mt f,m\phi)\le tm C_m\sup_{x \in X^\an} f(x)$$
	with 
	$$
	C_m:=((mL+M)^n)-((mL)^n). 
	$$
	By homogeneity of relative volumes, $\vol(mL,m\phi+mt f,m\phi)=m^{n+1}\vol(L,\phi+tf,\phi)$, thus
	$$
	m^{-n}C_m\inf_{x \in X^\an} f(x)\le \int_{X^\an} f\,(dd^c\phi+m^{-1}dd^c\psi_1)^n-t^{-1}\vol(L,\phi+t f,\phi)\le m^{-n}C_m\sup_{x \in X^\an} f(x),
	$$
	and hence
	$$
	\int_{X^\an} f\,(dd^c\phi+m^{-1}dd^c\psi_1)^n-m^{-n}C_m\sup_{x \in X^\an} f(x)\le\liminf_{t\to 0_+} t^{-1}\vol(L,\phi+t f,\phi)
	$$
	$$
	\le\limsup_{t\to 0_+} t^{-1}\vol(L,\phi+t f,\phi)\le \int_{X^\an} f\,(dd^c\phi+m^{-1}dd^c\psi_1)^n-m^{-n}C_m\inf_{x \in X^\an} f(x).
	$$
	Now $m^{-n}C_m\to 0$ as $m\to\infty$, and we conclude as desired 
	$$
	\lim_{t\to 0_+} t^{-1}\vol(L,\phi+tf,\phi)=\int_{X^\an} f\,(dd^c\phi)^n. 
	$$
	Let now $f$ be an arbitrary continuous function on $X^\an$. By density of model functions in $C^0(X^\an)$, we can pick a sequence $(f_i)_{i \in \N}$ of model functions on $X^\an$ such that 
	$$
	\e_i:=\sup_{x \in X^\an}|f(x)-f_i(x)|\to 0.
	$$
	Pick any ample line bundle $M$ on $X$. Since $M$ admits ample $\Q$-models on arbitrarily high models \cite[Proposition 4.11, Lemma 4.12]{GM}, each model function $f_i$ can be written as $f_i=\psi_{i1}-\psi_{i2}$ where $\psi_{i1},\psi_{i2}$ are model metrics on  $a_iM$ for some non-zero $a_i \in \N$, determined by ample $\Q$-models $\Mcal_{i1},\Mcal_{i2}$ of $a_iM$.
	
	Since $f_i-\e_i\le f\le f_i+\e_i$, the monotonicity of relative volumes yields for each $t>0$
	$$
	\vol(L,\phi+t f_i,\phi)-tV\e_i\le\vol(L,\phi+tf,\phi)\le\vol(L,\phi+tf_i,\phi)+t V\e_i
	$$
	with $V:=(L^n)$. By the first part of the proof, we infer
	$$
	\int_{X^\an} f_i\,(dd^c\phi)^n-V\e_i\le\liminf_{t\to 0_+} t^{-1}\vol(L,\phi+tf,\phi)
	$$
	$$
	\le\limsup_{t\to 0_+} t^{-1}\vol(L,\phi+tf,\phi)\le\int_{X^\an} f_i\,(dd^c\phi)^n+V\e_i,
	$$
	and letting $i\to\infty$ yields as desired
	$$
	\lim_{t\to 0_+} t^{-1}\vol(L,\phi+tf,\phi)=\int_{X^\an} f\,(dd^c\phi)^n.
	$$
	Replacing $f$ by $-f$, we conclude that the above holds also for $t$ negative, and Theorem~\ref{thm:main} follows.
\end{proof}
%
%
\subsection{Differentiability and orthogonality}
In this subsection, we assume that continuity of envelopes holds for $(X,L)$. The psh envelope $\env(\phi)$ of a continuous metric $\phi$ on $L$ is thus the greatest continuous psh metric on $L$ such that $\env(\phi)\le\phi$, see \S \ref{sec:psh}. Note that $\phi\mapsto\env(\phi)$ is monotone increasing, and satisfies $\env(\phi+c)=\env(\phi)+c$ for $c\in\R$, two properties that formally imply 
\begin{equation}\label{equ:lipP}
|\env(\phi)-\env(\p)|\le\sup_{x \in X^\an}|\phi(x)-\p(x)|
\end{equation}
for all continuous metrics $\phi,\p$ on $L$. 

To ease notation, we fix in what follows a reference continuous psh metric $\phi_0$ on $L$, and denote by
$$
\en(\phi):=\en(\phi,\phi_0)
$$
the relative energy of a continuous psh metric $\phi$ on $L$ with respect to $\phi_0$. By Theorem~\ref{thm:BE}, we have 
\begin{equation}\label{equ:BE}
\en(\env(\phi))=\vol(L,\phi,\phi_0)
\end{equation}
for all continuous metrics $\phi$ on $L$. 

\begin{defi} Given a continuous metric $\phi$ on $L$, we say that 
	\begin{itemize}
		\item \emph{$\en\circ\env$ is differentiable at $\phi$} if
		\begin{equation} \label{equ:diff}
		\frac{d}{dt}\bigg|_{t=0}\en(\env(\phi+t f))=\int_{X^\an}f\,(dd^c\env(\phi))^n. 
		\end{equation}
		for all $f\in C^0(X^\an)$ ; 
		\item \emph{orthogonality holds for $\phi$} if the Monge--Amp\`ere measure $(dd^c\env(\phi))^n$ is supported in the contact locus $\{\env(\phi)=\phi\}$, \ie 
		\begin{equation} \label{orthogonality property}
		\int_\Xan\left(\phi-\env(\phi)\right)(dd^c\env(\phi))^n=0. 
		\end{equation}
	\end{itemize}
\end{defi}

\begin{thm}\label{thm:difforth} Assume that continuity of envelopes holds for $(X,L)$. Then $\en\circ\env$ is differentiable at each continuous metric $\phi$ on $L$, and orthogonality holds for $\phi$. 
\end{thm}

\begin{lem}\label{lem:difforth} The following properties are equivalent:
	\begin{itemize}
		\item[(i)] $\en\circ\env$ is differentiable at all continuous metrics on $L$;
		\item[(ii)] $\en\circ\env$ is differentiable at all continuous psh metrics on $L$;
		\item[(iii)] orthogonality holds for all continuous metrics on $L$. 
	\end{itemize}
\end{lem}
\begin{proof} (i)$\Longrightarrow$(ii) is trivial. We reproduce the simple argument for (ii)$\Longrightarrow$(iii) given in~\cite[Theorem 6.3.2]{BGJKM}. Pick a continuous metric $\phi$, and set $\psi:=\env(\phi)$ and $f:=\phi-\psi$. For each $t\in[0,1]$, $\psi+t f=(1-t)\env(\phi)+t\phi$ satisfies $\env(\phi)\le\psi+t f\le\phi$, and hence $\env(\phi)=\env(\psi+tf)$. Differentiability of $\en\circ\env$ at $\p$ thus yields
	$$
	0=\frac{d}{dt}\bigg|_{t=0}\en(\env({\p} +tf))=\int_{X^\an} f\,(dd^c\env(\phi))^n=\int(\phi-\env(\phi))(dd^c\env(\phi))^n,
	$$
	which proves that $\phi$ satisfies the orthogonality property. Finally, the following simple argument for (iii)$\Longrightarrow$(i) is similar to the proof of \cite[Lemma 6.13]{LN}. Pick a continuous metric $\phi$ and a continuous function $f$. By concavity of $\en$ (see ~\eqref{equ:Econc}), we have
	$$
	\int_{X^\an}\left(\env(\phi+tf)-\env(\phi)\right)(dd^c\env(\phi+t f))^n\le\en(\env(\phi+tf))-\en(\env(\phi))
	$$
	$$
	\le\int_{X^\an}\left(\env(\phi+tf)-\env(\phi)\right)(dd^c\env(\phi))^n. 
	$$
	Using the orthogonality property at $\phi+tf$ and $\phi$ together with $\env(\phi)\le\phi$ and $\env(\phi+tf)\le\phi+t f$, this yields for $t>0$
	$$
	\int_{X^\an} f\,(dd^c\env(\phi+tf))^n\le\frac{\en(\env(\phi+tf))-\en(\env(\phi))}{t}\le\int_{X^\an} f\,(dd^c\env(\phi))^n,
	$$
	and hence 
	$$
	\lim_{t\to 0_+}\frac{\en(\env(\phi+tf))-\en(\env(\phi))}{t}=\int_{X^\an} f\,(dd^c\env(\phi))^n,
	$$
	by uniform convergence of $\env(\phi+tf)$ to $\env(\phi)$, cf.~\eqref{equ:lipP}. Replacing $f$ by $-f$ proves (iii)$\Longrightarrow$(i).
\end{proof}

\begin{proof}[Proof of Theorem~\ref{thm:difforth}] Taking into account~\eqref{equ:BE}, Theorem~\ref{thm:main} precisely says that $\en\circ\env$ is differentiable at every continuous psh metric on $L$, and Theorem~\ref{thm:difforth} thus follows from  Lemma~\ref{lem:difforth}.
\end{proof}
%
%
\subsection{An application to Monge--Amp\`ere equations}
In this subsection we still assume that continuity of envelopes holds for $(X,L)$. As in~\cite{trivval}, we define a (possibly singular) psh metric on $L$ as a decreasing limit of continuous psh metrics, not identically $-\infty$ on any component of $X$. A subset $E\subset X^\an$ is \emph{pluripolar} if there exists a psh metric $\phi$ with $\phi\equiv-\infty$ on $E$, this condition being easily seen to be independent of the choice of ample line bundle $L$. If $E$ is nonpluripolar, one proves exactly as in~\cite[Proposition 5.2(ii)]{nakstab} that for each continuous metric $\phi$ on $L$ there exists a constant $C>0$ such that 
\begin{equation}\label{equ:izunpp}
\sup_{x \in X^\an}(\psi(x)-\phi(x))\le\sup_{x \in E}(\psi(x)-\phi(x))+C
\end{equation}
for all psh metrics $\psi$ on $L$. Given a nonpluripolar compact $E\subset X^\an$ and a continuous metric $\phi$ on $L$, we can thus define the \emph{equilibrium metric} of the pair $(E,\phi)$ as
$$
\env(E,\phi):=\sup\{\text{$\psi$ psh metric on $L$}\mid\psi\le\phi\text{ on }E\}. 
$$
Since every psh metric $\p$ on $L$ is a decreasing limit of continuous psh metrics, Dini's lemma easily yields
$$
\env(E,\phi)=\sup\{\psi\text{ continuous psh metric on }L\mid\psi\le\phi\text{ on }E\},
$$
(see \cite[Proposition 7.26]{BE}) which is thus lsc. By~\eqref{equ:izunpp}, the family of metrics $\psi$ in the definition of $\env(E,\phi)$ is uniformly bounded from above, the usc regularization $\env(E,\phi)^\star$ is thus psh, since we assume continuity of envelopes (see \cite[Lemma 7.30]{BE}). As a result, $\env(E,\phi)^\star\le\phi$ holds on $E$ if and only if $\env(E,\phi)=\env(E,\phi)^\star$ is continuous. Following classical terminology in pluripotential theory, we then say that $(E,\phi)$ is \emph{$L$-regular}. 

For a nonpluripolar point $x\in X^\an$, $L$-regularity of $(\{x\},\phi)$ is independent of the continuous metric $\phi$, as the latter only appears through its value at $x$, and we then simply say that \emph{$x$ is $L$-regular}. 

\begin{exam} By~\cite[Lemma 2.20, Theorem 2.21]{trivval}, every quasimonomial point of $X^\an$ is nonpluripolar. 
\end{exam} 
Conjecturally, every nonpluripolar point should be $L$-regular; this has been shown in~\cite[Theorem 5.13]{nakstab} when $X$ is smooth, $K$ has residue characteristic $0$, and is trivially or discretely valued. 

Relying on the variational argument developed in~\cite{BBGZ,nama}, we prove the following result, which corresponds to Corollary C in the introduction. 

\begin{thm}\label{thm:MApoint} Assume that continuity of envelopes holds for $(X,L)$. Let $x\in X^\an$ be a nonpluripolar point, $\phi$ a continuous metric on $L$, and assume that $x$ is $L$-regular, so that
	$$
	\phi_x:=\env(\{x\},\phi)=\sup\{\psi\text{ psh metric on }L\mid \psi(x)\le\phi(x)\}
	$$
	is continuous and psh. Then $V^{-1}(dd^c\phi_x)^n=\d_x$ with $V:=(L^n)$. 
\end{thm}

\begin{proof} Pick $f\in C^0(X^\an)$. Since $\env(\phi_x+f)-f(x)$ is a continuous psh metric on $L$ and satisfies 
	$$
	\env(\phi_x+f)(x)-f(x)\le\phi_x(x)=\phi(x),
	$$
	we have $\env(\phi_x+f)-f(x)\le\phi_x$ by definition of the latter, and hence $\en(\env(\phi_x+f))-V f(x)\le\en(\phi_x)$. Applying this to $tf$, $t>0$, we infer
	$$
	t^{-1}\left(\en(\env(\phi_x+t f))-\en(\phi_x)\right)\le V f(x),
	$$
	and Theorem~\ref{thm:difforth} thus yields
	$$
	\int_{X^\an} f\,(dd^c\phi_x)^n\le V f(x).
	$$
	Replacing $f$ with $-f$ concludes the proof. 
\end{proof}


\bibliographystyle{alpha}
\bibliography{bibli}

\end{document}